\documentclass[12pt]{article}

\usepackage[usenames,dvipsnames]{xcolor}
\usepackage[left=1.3in, right=1.3in, marginparwidth=2in]{geometry} 

\usepackage{amsmath,amsthm,verbatim,amssymb,amsfonts,amscd,graphicx}
\usepackage[shortlabels]{enumitem}
\usepackage{mathtools}
\usepackage{hyperref}
\hypersetup{colorlinks, citecolor=Violet, linkcolor=NavyBlue, urlcolor=Blue}

\newtheorem{thm}{Theorem}[section]
\newtheorem{cor}[thm]{Corollary}
\newtheorem{prop}[thm]{Proposition}
\newtheorem{lem}[thm]{Lemma}
\newtheorem{defn}[thm]{Definition}
\newtheorem{assumption}[thm]{Assumption}
\theoremstyle{remark}
\newtheorem{rmk}[thm]{Remark}

\newcommand{\alias}[2]{\providecommand{#1}{} \renewcommand{#1}{#2}}
\newcommand{\F}{\mathcal{F}}
\newcommand{\pr}{\mathbb{P}}
\newcommand{\N}{\mathbb{N}}

\newcommand{\R}{\mathbb{R}}
\alias{\U}{\mathbb{U}}

\alias{\G}{\mathcal{G}}

\newcommand{\tsum}{{\textstyle\sum}}
\newcommand{\dsum}{{\displaystyle\sum\nolimits}}
\newcommand{\tl}{{\tilde{l}}}
\newcommand{\tq}{{\tilde{q}}}
\newcommand{\bt}{\bar{t}}
\newcommand{\bx}{\bar{x}}
\newcommand{\by}{\bar{y}}
\newcommand{\bz}{\bar{z}}
\newcommand{\sK}{{\mathcal K}}
\newcommand{\sL}{{\mathcal L}}

\newcommand{\sP}{{\mathcal P}}
\newcommand{\sS}{{\mathcal S}}

\newcommand{\sV}{{\mathcal V}}
\newcommand{\PP}{{\mathbb P}}

\newcommand{\LL}{{\mathbb L}}
\newcommand{\FF}{{\mathbb F}}

\newcommand{\one}{1}
\newcommand{\ind}[1]{ \one_{{#1}}}

\newcommand{\lmr}[3]{{{}^{#1}{#2}^{#3}}}
\DeclareMathOperator{\dom}{dom}

\DeclareMathOperator{\Cl}{Cl}
\newcommand{\bmo}{\mathrm{bmo}}
\newcommand{\BMO}{\mathrm{BMO}}

\DeclareMathOperator{\Tr}{Tr}

\newcommand{\tot}{\frac{1}{2}}

\newcommand{\Bset}[1]{\Big\{#1\Big\}}
\newcommand{\set}[1]{\{#1\}}

\newcommand{\sq}[1]{\set{#1}_{n\in\N}}
\newcommand{\abs}[1]{| #1 |}

\newcommand{\ab}[1]{\langle #1 \rangle}

\newcommand{\norm}[1]{{||#1||}}
\newcommand{\eps}{\varepsilon}
\newcommand{\ld}{\lambda}
\newcommand{\toucp}{\stackrel{\mathrm{ucp}}{\rightarrow}}

\newcommand{\rupk}[1]{{#1}^{k}}

\newcommand{\rxk}{\rupk{X}}
 
 \newcommand{\ryk}{\rupk{Y}}

 \newcommand{\Zk}{\rupk{Z}}

\newcommand{\fk}{\rupk{f}}

\newcommand{\define}[1]{{\textbf{#1}}}

\newcommand{\eand}{\text{ and }}
\newcommand{\ewhere}{\text{ where }}

\newcommand{\ito}{It\^ o}
\newcommand{\itos}{It\^ o's}

\newcommand{\sam}{\{a_m\}}

\usepackage[textsize=tiny, color=white, linecolor=black!20,
bordercolor=black!20]{todonotes}

\numberwithin{equation}{section}

\newcommand{\lop}{(\ld \otimes \pr)}
\newcommand{\lp}[1]{ \lop  [ #1 ] }
\newcommand{\lpB}[1]{ \lop \Big[ #1 \Big] }

\newcommand{\rt}{[0,T]}
\newcommand{\rx}{\R^d}
\newcommand{\ry}{\R^n}
\newcommand{\rz}{\R^{n\times d}}

\newcommand{\subf}{S^f}

\newcommand{\tphi}{\tilde{\phi}}
\begin{document}

\title{A
Characterization of
Solutions of Quadratic BSDEs and a New Approach to Existence\footnote{During the preparation of this work the second author has been supported by the National Science Foundation under Grant No. DSM1815017 (2018-2021). Any opinions, findings and conclusions or recommendations expressed in this material are those of the author(s) and do not necessarily reflect the views of the National Science Foundation (NSF).}}

\author{
Joe Jackson\thanks{Department of Mathematics, The University of Texas at Austin, Austin, TX \texttt{jjackso1@utexas.edu}} \and Gordan \v{Z}itkovi\'c\thanks{Department of Mathematics, The University of Texas at Austin, Austin, TX \texttt{gordanz@math.utexas.edu}}%
}


\maketitle

\begin{abstract}
    We provide a novel characterization of the solutions of a quadratic BSDE, which is analogous to the characterization of local martingales by convex functions. We then use our main result to show that BSDE solutions are closed under ucp convergence. Finally, we use our closure result obtain a sufficient condition for existence, and discuss specific cases in which this sufficient condition can be verified. 
\end{abstract}

\section{Introduction}
Backward Stochastic Differential Equations (BSDEs) were first introduced by
Bismut \cite{Bis73} in 1973 in the linear case, and later generalized by Pardoux
and Peng \cite{Pardoux-Peng}, who studied them under general Lipschitz
conditions.  In this paper we are concerned with quadratic BSDEs, i.e., BSDEs of
the form 
\begin{align}
\label{bsde-intro}
	dY = f(t,Y,Z)\, dt + Z\, dB,
\end{align}
with $Y$ taking values in some $\R^n$ (systems of BSDEs), $Z$ in $\R^ {n\times
d}$,  whose driver $f$ has quadratic growth in $z$. In 2000, Kobylanski \cite{Kobylanski}
established the well-posedness of such one-dimensional quadratic BSDEs
with a bounded
terminal condition.
For a  survey of the most important results in this active field up to
2017, we refer the reader to  the textbook \cite{Zhang} of Zhang. 

Well-posedness for multi-dimensional quadratic BSDEs has
proved more challenging and is still a central open problem. The most general
results in this direction come   from Tevzadze \cite{Tevzadze} under a smallness
assumption, and from Xing and \v Zitkovi\' c \cite{xing2018} in the Markovian
setting. We refer the reader to 
\cite{HuTan16,JamKupLuo17,CheNam17,Nam19,HarRic19} and the references
therein for some other important contributions. 

\medskip

The importance of BSDEs in optimal stochastic control and stochastic game theory is
well documented  (see, e.g., the monographs  \cite{Pha09} or \cite{Car16} and
the references therein). It is  not as well known that quadratic BSDEs also
appear in stochastic differential geometry, as pioneered by Darling in
\cite{Darling}, based on the earlier work of Kendall, Picard and others (see,
e.g.,  \cite{Ken90, Ken94, Ken91b, Ken92a, Pic94} and the references therein). 
Let $\Gamma$ be an affine connection on a differentiable manifold $M$ which, for
the sake of simplicity, we take to be diffeomorphic to $\R^n$. Martingales on
$M$ (also called $\Gamma$-martingales) are manifold-valued processes  that 
generalize continuous local martingales and  share many of
their properties 
(see, e.g., \cite[Chapter IV]{Eme89} for an
overview). One such property is the following:  an $M$-valued process $Y$ is a
$\Gamma$-martingale if and only if $\phi(Y)$ is a submartingale whenever $\phi$
is a (locally defined) convex function (see \cite[Theorem 4.9, p.~43]{Eme89}
for a precise statement). As in the Euclidean case, this characterization gives
an analytic insight into various properties of martingales,  and can be used to
prove, among others, the fundamental result that the set of $\Gamma$-martingales
is closed under the ucp (uniform on compacts in probability)
convergence (see \cite[Theorem 4.43, p.~47]{Eme89}). 

The relationship between $\Gamma$-martingales and BSDEs stems from  the
simple observation that a process $Y$, adapted
with respect to a Brownian filtration and  taking values in $M$, is a
$\Gamma$-martingale if and only if it  solves the BSDE 
\begin{align} \label{martdriver}
	dY = f(Y,Z)\, dt + Z\, dB \text{ with }
	    f^i(y,z) = \frac{1}{2} \sum_{i_1,i_2 =1}^n \sum_{j = 1}^d
	    \Gamma^{i}_{i_1,i_2}(y) z^{i_1}_j  z^{i_2}_{j},
\end{align}
where $n$ is the dimension of the manifold, $d$ is the dimension of the driving
Brownian motion, and $\Gamma^{i}_{i_1, i_2}$ are the Christoffel symbols of the
affine connection on $M$.  Using this observation, Darling constructed
$\Gamma$-martingales with prescribed terminal conditions and studied their
properties.  In his analysis, he used the fact that the family of
$\Gamma$-martingales is closed under ucp convergence  as a key tool in
addressing the well-posedness of BSDEs of the special form \eqref{martdriver}
under geometric restrictions on the Christoffel symbols $\Gamma$. We note that the equation \eqref{martdriver} can be viewed as a BSDE on a manifold, written in local coordinates. This perspective is developed in \cite{blache1} and \cite{blache2}, where more general BSDEs on manifolds are treated, and the results of \cite{Darling} for the equation \eqref{martdriver} are extended considerably.

\subsection{Our results} Inspired by the connection between special
BSDEs and $\Gamma$-mar\-tin\-ga\-les
considered by Darling, we set out to investigate the following questions: In
which ways do solutions to \emph{general} BSDEs behave like martingales? Just as
Darling used a special class of BSDEs to describe martingales on manifolds, can
we tie a properly generalized notion of a martingale to a general
BSDE? Is there a geometric picture associated with such a generalization? 

Being the starting point in this program, the present paper answers a more
specific question: Is there a ``convex-function'' characterization of solutions
to BSDEs in the spirit of Emery? Indeed, since martingales are
characterized by the way they transform
under convex maps, can we do the same with solutions of BSDEs?  Of course,
classical convex functions -  which turn each martingale into a
local submartingale -
will not work for a BSDE with a general driver $f$.  We show that the suitable
notion is that of $f$-subharmonicity and, given an \ito{}-process $Y$, we give
the following characterization:  there exists a process $Z$ such that $(Y,Z)$
solves \eqref{bsde-intro} if and only if  $\phi(t,B,Y)$ is a submartingale
whenever $\phi$ is an $f$-subharmonic function. In the case $f\equiv 0$,
$f$-subharmonic functions coincide with $C^{1,2,2}$-convex functions,
while in the
``geometric'' case of Darling, they coincide with geodesically convex functions
on the manifold. The general case is considerably more complicated than either
of the two special cases mentioned above and corresponds to the class of
sufficiently smooth subsolutions to a fully nonlinear parabolic PDE.  We note
that our notion of $f$-subharmonicity resembles the notion of a Lyapunov
function introduced in \cite{xing2018}, but it is not strictly comparable to it
and its use is quite different. We note also that Proposition 2.4.1 of \cite{blache1} is similar in spirit to our Theorem \ref{main}, although in a much more restrictive geometric setting, where Emery's proof that martingales on manifolds are characterized by convex functions is easier to adapt. The supermartingale property under an
exponential transformation in dimension $n=1$ has been used by \cite{BarElk13}
to characterize certain classes of semimartingales and prove powerful closure
results under monotone convergence. These results also lead to an alternative
proof of  a part of Kobylanski's theorem. 

One of the most important features of $f$-subharmonic functions is that they
naturally come defined only locally, and that local $f$-subharmonic functions
cannot always be extended to a larger (full) domain. This property is well 
known
in the case of geodesically-convex functions on certain Riemannian manifolds,
even those diffeomorphic to a Euclidean space.  Interestingly enough, this
feature kicks in only in dimensions higher than $1$, and only with
drivers
exhibiting a truly quadratic growth in $z$. It also explains why our
characterization needs to be local in nature - sometimes, quite simply, there
are no non-trivial globally defined $f$-subharmonic functions. 

Dimensionality also plays a role when it comes to the most basic  
features of processes that become submartingales
under $f$-subharmonic transformations. In dimension $1$, such
processes are automatically \ito{}-processes,
but in higher dimensions additional structure is needed to rule out
the presence of singular drift components (see Corollary \ref{cor:ab}).  


\smallskip

After proving Theorem \ref{main}, we use our characterization to prove that, under appropriate conditions, the set
of processes $Y$ which solve
\eqref{bsde-intro} (for some terminal condition) is closed under ucp
limits. We note that this closure result holds in a high degree of generality when $n = 1$, but some additional structure is needed when $n > 1$, see Remark \ref{rmk:dimone}. This result, namely Theorem \ref{ucplimits},
is directly analogous to the parallel result about 
$\Gamma$-martingales on manifolds mentioned above. 

Finally, we present in sub-section 3.2 a ``template for existence", i.e. a sufficient condition for the existence of solutions to quadratic BSDE systems. This result, namely Proposition \ref{template}, shows how the geometric arguments in \cite{Darling} might eventually be applied to general quadratic systems, and its proof is based on the ucp closure result Theorem \ref{ucplimits}. We note that it is possible to verify the hypotheses of Proposition \ref{template} in certain special cases, see Remark \ref{rmk:template}.  

\subsection{Organization of the Paper.}\
In Section 2, we describe notation and assumptions, introduce the notion of
$f$-subharmonic functions and $f$-local martingales, and then state and prove
the main result of the paper. In Section 3 we show ucp-stability
solutions of BSDEs, and then
present a template for existence.   

\subsection{Notation and conventions.}
Throughout the paper, we work on a a fixed probability space $(\Omega, \FF,
\pr)$ which hosts a $d$-dimensional Brownian motion $B$ with a deterministic
time horizon $T < \infty$. The filtration $\FF = (\F_t)_{t
\in [0,T]}$ is the augmented filtration of $B$, while  the \emph{augmented}
filtration generated by a process $X$ is denoted by $\FF^X$.

Components of vectors in $\R^n$ get
indices based on $i$ (such as $i$ or $i'$), written as
superscripts. For $\R^d$-valued vectors, indices are based on $j$ and
are written as subscripts. Vectors in 
$\R^{n\times d}$ are interpreted as $n\times
d$-matrices, and we use indexing of the form $z=(z^i_j)$,
with the understanding that 
$z^i$ denotes the $i$-th row and $z_j$
$j$-th column. Indices based on $i$ always range from $1$ to $n$, and those
based on $j$ from $1$ to $d$. This way, for example, $\tsum_i = 
\tsum_{i=1}^n$ and $\tsum_{i',j} = \tsum_{i'=1}^n \tsum_{j=1}^d$. 
Finally, we \emph{do not} use the Einstein summation convention. 

The standard inner product $\ab{\cdot,\cdot}$ is used
on $\R^n$ and $\R^d$. In the matrix (i.e., $\R^{n\times d}$) case, we
use the Frobenius inner
product given by
$\ab{z, z'} = \sum_{i,j} z^i_j z'^i_j = \Tr 
 z^T z'$. In all three cases, we use the induced
norm,
denoted by $\abs{\cdot}$. 

$\sP$ denotes the set of all $\FF$-progressively measurable $\R^{n \times
d}$-valued processes, while $\sP^p$ is the set of all $Z\in \sP$ such that $
\int_0^T \abs{Z_t}^p\, dt<\infty$, a.s.  For $Z\in \sP^2$, $\int Z
\cdot dB$ is $\ry$-valued process defined by  $(\int Z\cdot dB)^i
\coloneqq \sum_{j} \int Z^i_{j} dB_j\ $. $\sS^p$ denotes the set of
all continuous adapted $\R^n$-valued processes $Y$ with $
\norm{Y}_{\sS^p}\coloneqq \norm{Y^*_T}_{\LL^p}<\infty$. $\bmo$
denotes the set of all processes
$Z\in \sP^2$ such that $\int Z\cdot dB$ is a BMO martingale, and
$\norm{Z}_{\bmo} \coloneqq \norm{\int Z\cdot B}_{\BMO}$, where $
\norm{\cdot}_{\BMO}$ denotes the standard $\BMO$ norm. 

\section{A characterization of solutions}
This section gives a novel characterization the class of processes $Y$ which
solve a given BSDE. A central role in this
characterization will be played by the class of $f$-subharmonic functions,
introduced below.

\subsection{Basic assumptions.}
Given a continuous \define{driver} $f:\rt\times\ry\times\rz\to\R^n$, a  pair of
processes $ (Y,Z)$ is said to be a \define{solution} to
\begin{equation} \label{bsde}
    Y^i_T = Y^i_t - \int_t^T f^i(s,Y_s,Z_s)\, ds + \int_t^T Z^i_s
    \cdot dB_s,\ 1\leq i \leq n,
\end{equation}
if $Y^i$ is a continuous $n$-dimensional $\FF$-semimartingale, $Z\in \sP^2$ and
\eqref{bsde} holds for all $1 \leq i \leq n$, and all  $t\in [0,T]$, a.s. Note
that our solution concept does not come with an a-priori terminal condition or
any regularity imposed on $Y$ or $Z$ beyond the bare minimum. In
particular, $\int Z\cdot dB$ is not necessarily a (true) martingale. 

In this paper we will abuse terminology and say that $Y$
(alone) is a \define{solution} to \eqref{bsde} if there exists $Z\in
\sP^2$ such that
$(Y,Z)$ is a solution in the sense described above. 

The following assumption can be viewed as a quantitative version of the requirement that $z \mapsto f(t,y,z)$ has quadratic growth for each $t$ and $y$.
 \begin{assumption}
\label{assumptions}

The driver $f:\rt\times\ry\times \rz \to \ry$ is jointly continuous and
there exists a continuous increasing function $\kappa : [0,\infty) \to [0,\infty)$ such that 
for all $(t,y) \in [0,T] \times \R^n$ and $z, z' \in \R^{n \times d}$, we have
\begin{align*}
       \abs{ f(t,y,z') - f(t,y,z) } 
       \leq 
       \kappa(|y|)\big(1 + \abs{z} + 
       \abs{z'}\big)
         \abs{z'-z}
\end{align*}
\end{assumption}

\subsection{The notion of $f$-subharmonicity.}
When $Y$ is a solution to \eqref{bsde} and $\phi : \rt\times\rx\times\ry \to \R$
is a $C^{1,2,2}$-function, \itos{}
formula implies that 
\begin{align}
  \label{ito}
  d \phi(t,B_t,Y_t) =  \sL^f \phi (t,B_t, Y_t; Z_t)\, dt + dL_t
\end{align}
where $L$ is a local martingale and
$\sL^f\phi : \rt\times\rx\times\ry\times\rz \to \R$, 
is defined by 
\begin{equation} \label{opdef}
  \begin{split}
      \sL^f \phi (t,x,y; z) &= -\dsum_i
      \phi_{y^i}(t,x,y)f^i(t,y,z)\\ & +
      \phi_t
      (t,x,z) + \frac{1}{2} \dsum_{i, i'} \phi_
      {y^{i} y^{i'}}(t,x,y) \ab{z^{i}, z^{i'}}  \\
      & + \dsum_{i,j} \phi_{x_j y^i}(t,x,y) z^i_j  + \frac{1}{2} 
      \dsum_{j}  \phi_{x_j x_j} (t,x,y).
\end{split}
\end{equation}
Therefore, if $\phi$ is a $C^{1,2,2}$-function with the
  property that $\sL^f \phi \geq 0$ everywhere, the composition 
$\phi(t,B_t,Y_t)$ is necessarily a local submartingale.
A local version of this property motivates the 
the following definition:
  \begin{defn}[$f$-subharmonic functions] \label{sub-def}
    A real valued $C^{1,2,2}$-function $\phi$ defined on
      an open set $\dom \phi$ of $\rt\times\rx\times\ry$
    is said to be \define{$f$-subharmonic} if 
    \[ \inf_{z\in \R^{n \times d}} \sL^f\phi (t,x,y; z) \geq 0 \text{ 
    for all $(t,x,y) \in \dom \phi$.}\] 
  \end{defn}
\noindent The set of all $f$-subharmonic functions is denoted by
$\subf$.

  \begin{rmk}
    Probabilistically, one can interpret $f$-subharmonicity
    as a sufficient condition for a function to transform any
    solution $Y$ of \eqref{bsde} to a submartingale ``while
    $(t,B_t,Y_t) \in \dom \phi$".
  To understand $f$-subharmonicity 
  better from the analytic point of view, we consider only the case where $\phi$
  does not depend on $t$ and start in the simplest setting 
  with $f=0$ and $n= d = 1$ in
  which 
  \[ \sL^0 \phi (t,x,y;z) = \frac{1}{2} \langle \zeta,   D^2
  \phi(t,x,y) \zeta \rangle \ewhere \zeta = (1,z) \eand \]
  $D^2 \phi(t,x,y)$ is the Hessian matrix of $\phi$ in 
  variables $x$ and $y$. It follows immediately
  that $\phi$ is $f$-subharmonic if and only if 
  $D^2 \phi$ is non-negative definite on $\dom \phi$, 
  which is, in turn, equivalent to joint convexity of $\phi$ in $x$
  and $y$ on $\dom \phi$. 
  
  When $n =1$ but $d > 1$, $f$-subharmonicity, even for $f=0$, no longer implies convexity.
  Indeed, in that case we have
  \[ \sL^0 \phi (t,x,y;z) =  \tot  \langle \xi, H \xi
    \rangle,\]
    where $\xi = (\abs{z}, z_1/ \abs{z}, \dots, z_n/\abs{z})$ (with $z_i/\abs{z}\coloneqq0$ when $\abs{z}=0$)  and 
    the matrix $H$ is given by
    \[ H = \begin{pmatrix} 
    \phi_{yy}  & \phi_{y x_1} & \phi_{ y x_2} & \dots & \phi_{y x_d} \\
    \phi_{x_1 y} & \Delta_x \phi & 0 & \dots & 0\\ 
    \phi_{x_2 y} &  0 & \Delta_x\phi &  \dots & 0\\ 
    \vdots & \vdots & \vdots & \ddots & \vdots \\
    \phi_{x_d y} &  0 & 0 & \dots & \Delta_x \phi \\ 
  \end{pmatrix}. \]
As above, $f$-subharmonicity is equivalent to non-negative
definiteness of the matrix $H$, but $H$ is no longer the Hessian of
$\phi$. It is, however, obtained from the Hessian $D^2 \phi$ by replacing the
submatrix of $(x_j, x_{j'})$-derivatives by a $\Delta_x \phi$-multiple
of the $n\times n$ identity matrix. When $\phi$ does not depend on
$y$, $0$-subharmonicity is equivalent to classical subharmonicity of
$\phi$, a property strictly weaker than convexity. In general, the
notion of $0$-subharmonicity can be interpreted as a ``convex
combination" of convexity and subharmonicity.

The general case, when $f$ does not necessarily vanish, is much
harder to interpret, primarily because the fact that the variable
$z$, which we used to test positive definiteness above no longer
separates from the rest. One special case where such a separation
does occur is when $f(t,x,y,\cdot)$ is a quadratic
form in $z$, i.e., 
\[ f(t,x,y,z) = \sum_{i,i'} F_{i, i'}(t,x,y) z^i z^{i'},\]
a particular case of 
which we have seen in \eqref{martdriver} in the Introduction. 
We hope it may shed some light on a possible deeper geometric
meaning of $f$-subharmonicity; it also  leads to non-negative
definiteness, but of 
a matrix $H'$ which is obtained from $H$ above by adding
certain first-order terms to it, just as in the
the coordinate representation of the covariant Hessian on a
differentiable manifold with an affine connection. 
\end{rmk}

  \subsection{The main result}
With the family $\subf$ of all
$f$-subharmonic functions introduced in Definition \ref{sub-def}
above, we give two more definitions whose primary purpose is to
improve presentation of our main theorem and its proof. Given a stochastic process $\Phi$ and a stopping time $\tau$, we denote by $\Phi^{\tau}$ the process $\Phi$ stopped at time $\tau$, i.e. $\Phi^{\tau}_t = \Phi_{\tau \wedge t}$. For two
stopping times $\tau_1\leq \tau_2$ and a stochastic process $\Phi$, we
define \[ \lmr{\tau_1}{\Phi_t}{\tau_2} \coloneqq \Phi^{\tau_2}_{t
\vee \tau_1 } -
  \Phi^{\tau_2}_{\tau_1} = 
  (\Phi_{t} - \Phi_{\tau_1}) 1_{[\tau_1, \tau_2)}(t) + 
(\Phi_{\tau_2} - \Phi_{\tau_1}) 1_{[\tau_2, T]}(t) =
\int_0^t \ind{ (\tau_1,\tau_2]}(u)\, d\Phi_u.
\]
We say that $Y$ is a \define{local submartingale on
$[\tau_1,\tau_2)$} if $\lmr{\tau_2}{Y}{\tau_1}$ is a local
submartingale. 
\begin{defn}[Local $f$-martingales] \ 
  \begin{enumerate}
    \item Given a function $\phi$ defined on an open subset $U$ of
  $[0,T]\times \R^m$, and an $\R^m$-valued continuous adapted process 
  $X$, we say that
  $Y_t=\phi(t,X_t)$ is \define{a local submartingale 
  while $(t,X)\in U$}
  if it is a local submartingale on each stochastic interval
  $[\tau_1,\tau_2)$ such that $(t,X_t) \in U$ for all $t\in
  [\tau_1,\tau_2)$, a.s. 
\item 
    A continuous adapted process $Y$ is called a
    \define{local $f$-martingale} if
      $\phi(t,B_t, Y_t)$ is a
      local submartingale while $(t,B,Y)\in \dom \phi$
      for each $\phi \in \subf$.
  \end{enumerate}
  \end{defn}
  The main result of the paper is the following characterization:
\begin{thm}[Main]
  \label{main}
  Under Assumption \ref{assumptions}, an \ito{} process $Y$ is a local $f$-martingale
  if and only if it solves the BSDE \eqref{bsde}.
\end{thm}
With the proof left for the following section, here are some comments
on the form, scope and conditions of Theorem \ref{main}. 
\begin{rmk}\ 
  \label{after-main}
  \begin{enumerate}
  \item 	As we show in Corollary \ref{cor:ab} below, 
  the a-priori assumption that $Y$ is an \ito{} process is not
  necessary if either $n=1$ or $f$ satisfies an additional structural condition; a local $f$-martingale is automatically a
  semimartingale and an \ito{}-process in that case. 
  An inspection of our proof
  of Theorem \ref{main} below reveals that when $n>1$, 
  local $f$-martingales correspond to possibly \emph{generalized} solutions of 
  \eqref{bsde} where a singular finite-variation process $A^i$ is
  added to each equation and treated as component of the solution. 
  \item 
  Theorem \ref{main} would remain valid if we replaced 
      the Brownian motion $B$ in the definition of $f$-subharmonic
  function by a Markov process in an appropriate class.
  Of course, the definition of an $f$-subharmonic
  function would have to be suitably modified. 
  We leave the problem of
  identifying the range of possible replacements for $B$ 
  for further research. 

  Clearly, full $\omega$-dependence cannot be expected even under
  strong adaptivity assumptions without  changing
  the nature of our result in a fundamental way. Indeed, it
  would require replacing the notion of an $f$-subharmonic 
  function by a random field, and, as such, could not be used
  to separate analytic and probabilistic aspects of BSDEs. 
\item 
  The notion of a solution of a BSDE at our level of generality can be
  thought of as a
  particular form of dependence between a process $Y$ and the
  given Brownian motion. Therefore, the variable $B$, or some
  surrogate, cannot be left out completely out of the definition of
  an $f$-subharmonic function. Indeed,  consider the case $n=1$, 
  $f(t,y,z) = z$, and the class $S\subseteq \subf$ consisting of all
  $\phi$ in $\subf$ that do not depend on $x$. 
  Whether or not a
  process $\phi(t,Y_t)$ 
  is a local submartingale for each $\phi\in S$ depends only on the
  distribution of $Y$, so it will be enough to exhibit two processes
  equal in distribution, such that one
  solves
  \eqref{bsde}, but the other does not. For example, we may simply
  take
  \[ Y_t = t+B_t \eand Y'_t = t + \int_0^t h_s\, dB_s,\]
  where $h$ is any $\set{-1,1}$-valued progressive process not 
  equal to $1$, $\lop$-a.e.
  It is clear that both of these are Brownian motions with a unit
  drift, and that the first one solves \eqref{bsde} with $Z=1$.
  The other one is not a solution since the only candidate
  for the  second component of the solution, namely
  $Z'_t = h_t$, does not satisfy \eqref{bsde}. 
 \end{enumerate} 
\end{rmk}

\subsection{Proof of Theorem \ref{main}.}\label{sse:proof} In this subsection, we provide a proof of Theorem \ref{main}. We then show how the hypotheses Theorem \ref{main} can be weakened when $n = 1$ or $f$ satisfies an additional structural condition. 
  We fix a driver $f$
  which satisfies Assumption \ref{assumptions},  and start with some auxiliary results. 

\medskip

Given  $\bz \in \R^{n \times d}$,  we will call a function of the form
$g(z) = a_0 + b_0 \abs{ z- \bz} + c_0 \abs{z - \bz}^2$
with $c_0\geq 0$
\define{cone-quadratic about $\bz$}, and a function of
the form $h(z) = d_0 + e_0 \abs{z - \bz}^2$ with $e_0 \geq 0$ 
\define{purely
quadratic about $\bz$}. We identify such functions with points in $\R^2 \times
[0,\infty)$ and $\R \times [0,\infty)$, respectively, defined by their
coefficients. This way we
can speak about neighborhoods around cone-quadratic or purely quadratic
functions in the following, simple, lemma whose proof we leave to the reader:
\begin{lem} \label{quad}
  For any $\eps>0$, $\bz\in\R^{n \times d}$ and  a cone-quadratic function $l$ 
  about $\bz$ there
  exists a purely quadratic function $q$ about $\bz$ and neighborhoods $Q$
  around $q$ and $L$ around $l$ such that
  \begin{enumerate}
    \item $\tl(z) \leq \tq(z)$ for all $z\in\R^n$, $\tl \in L$ and
      $\tq
      \in Q$, 
    \item $q(\bz) \leq l(\bz) + \eps$
  \end{enumerate}
Furthermore, $q$ can be chosen non-constant. That is, we can take $q(z) = d_0 + e_0|z - z_0|^2$ with $e_0 > 0$.  
\end{lem}

The next lemma shows that $\phi$-subharmonic functions
exist locally, even when their behavior at a given point is
additionally restricted.
\begin{lem} \label{mainlem}
  For any $(\bt,\bx,\by,\bz)\in \rt\times\rx\times\ry\times\rz$,
  $1\leq i_0 \leq n$ 
  and $\epsilon > 0$, there
  exists an $f$-subharmonic function $\phi$ such that
  \begin{enumerate}
    \item $(\bt,\bx,\by)  \in \dom \phi$, 
    \item $\sL^f\phi(\bt,\bx,\by;\bz) \leq \epsilon$ and 
    \item $\phi_{y^{i_0}} (\bt,\bx,\by) = 1$ (or $-1$ if desired),
    and
    \item $\phi_{y^i} (\bt,\bx,\by) = 0$ for $i \ne i_0$.
    \end{enumerate}
\end{lem}

\begin{proof}
  We make the following Ansatz
  \begin{align}
    \label{phi-form}
\phi(t,x,y) = -\sum_{ i \neq i_0} y^{i} + \dsum_{i}   \frac{1}{\beta^i} E^i(x,y) +
\theta |x|^2
\end{align}
where
\begin{align*}
    E^i(x,y) = e^{\beta^i (y^i - \by^i) + \ab{\gamma^i, x -
    \bx}},
\end{align*}
and the constants $\theta\in\R$, $\beta \in \R^n$ and 
$\gamma \in \R^{n\times d}$,
as well as the domain $\dom \phi$ will be determined later.
  It follows by direct computation that
  \begin{align*}
      \phi_{y^i}(t,x,y) &= E^i(x,y)  + \delta^{ii_0} -1, &   \phi_{y^iy^i}
      (t,x,y) &= \beta^i E^i(x,y), \\
    \phi_{x_j x_j} (t,x,y) &= \dsum_i\frac{(\gamma^i_j)^2}
    {\beta^i} E^i(x,y) + 2\theta, &
    \phi_{y^i x_j}(t,x,y) &= \gamma^i_{j} E^i(x,y),
  \end{align*}
  where $\delta^{\cdot\cdot} = \ind{\cdot = \cdot}$ is the Kronecker
  delta function,
  so that
  \begin{equation}
  \label{lcomp}
  \begin{split}
  \sL^f \phi(t,x,y;z) &=  - \dsum_{i} \big(E^i(x,y) +\delta^{ii_0}  - 1
  \big)f^i(t,y,z) + \frac{1}{2} \dsum_{i} \beta^i E^i(x,y) |z^i|^2  
  \nonumber \\   &\quad + \dsum_{i,j} \gamma^i_{j} E^i(x,y) z^i_j + 
  \frac{1}
  {2} \dsum_{i,j}
  (\gamma^i_{j})^2\frac{1}{ \beta^i} E^i(x,y) + \theta d.
  	\end{split}	
  \end{equation}
Next, we define the function $l:[0,T]\times \R^{d} \times \R^n \times
\R^{n \times d} \to \R$ by  
  \begin{equation}
  \label{ldef}
  \begin{split}
 l(t,x,y,z) &\coloneqq 
  \dsum_i
  |E^i(x,y) + \delta^{ii_0} - 1|\big( \abs{f^i(t, y, \bz) - 
f^i(\bt, \by, \bz)}  \\ 
  & \quad + C(\kappa(|y|) + 1) \big(\abs{z - \bz } + \abs{z - \bz }^2 \big)  + 
  \dsum_{i} (E^i(x,y) + \delta^{ii_0} - 1) f^i(\bt, \by, \bz),  
  \end{split}
  \end{equation}
where $C>0$ is a constant such that
\begin{multline}
  \label{fbnd}
  f^i(t,y,z) \leq 
  f^i(\bt, \by, \bz)  + \abs{f^i(t, y, \bz) - 
    f^i(\bt, \by, \bz)} 
     +C\kappa(|y|) \big(\abs{z - \bz } + \abs{z - \bz }^2 \big)
\end{multline}

for all $t,y,z$ and $i$.
Such a choice of $C$  - which exists thanks to  
Assumption \ref{assumptions} and repeated application of the
triangle inequality - leads to the following tight inequality:
  \begin{equation}
    \label{lineq}
    \begin{split}
  l(t,x,y,z) &\geq \tsum_i (E^i(x,y) + \delta^{ii_0} - 1) f^i(t,y,z), 
  \text{for all $t,x,y,z$, and
  }\\
  l(\bt, \bx, \by, \bz) &= \tsum_i (E^i(\bx, \by) + \delta^{ii_0} - 1) f^i
  (\bt, \by, \bz).
\end{split}
\end{equation}
The $(\bt,\bx,\by)$-section of the function $l$ is a cone-quadratic
function of $z$ about $\bz$; Lemma 
\ref{quad} applied to it yields a purely quadratic function 
$q(z) = d_0 + e_0 \abs{z-\bz}^2$ with $e_0 \neq 0$,  a neighborhood $Q$ around it and a
neighborhood $L$ around $l(\bt,\bx,\by,\cdot)$. 

We now choose $\beta^i, \gamma^i$, and $\theta$ so that 
\begin{align*}
    \tot \beta^i = e_0, \quad
    \gamma^i_{j} = -2e_0 \bz^{i}_j, \text{ and } 
   \frac{1}
  {2} \dsum_{i,j}
  (\gamma^i_{j})^2\frac{1}{ \beta^i} + \theta d = d_0 + e_0 \abs{\bz}^2.
\end{align*}
With this choice of constants, we have
\begin{align*}
   \frac{1}{2} \dsum_{i} \beta^i E^i(\bx,\by) |z^i|^2  
  + \dsum_{i,j} \gamma^i_{j} E^i(\bx,\by) z^i_j + 
  \frac{1}
  {2} \dsum_{i,j}
  (\gamma^i_{j})^2\frac{1}{ \beta^i} E^i(\bx,\by) + \theta d = q(z).
\end{align*}
The coefficients of $l$ are continuous as functions of $(t,x,y)$, so
there exists a neighborhood $U$ of $(\bt,\bx,\by)$ such that
$l(t,x,y,\cdot) \in L$, for all $(t,x,y)\in U$. Since
$E^i(\bx,\by)=1$ for each $i$, we can shrink this
neighborhood further, if needed, to guarantee that the map
\begin{align*}
    z \mapsto \frac{1}{2} \dsum_{i} \beta^i E^i(x,y) |z^i|^2  
  + \dsum_{i,j} \gamma^i_{j} E^i(x,y) z^i_j + 
  \frac{1}
  {2} \dsum_{i,j}
  (\gamma^i_{j})^2\frac{1}{ \beta^i} E^i(x,y) + \theta d.
\end{align*}
lies in $Q$ as soon as $(t,x,y) \in U$. 

If we declare $\dom \phi = U$ 
the first conclusion of Lemma \ref{quad} implies that $\phi$ is
$f$-subharmonic with $(\bt,\bx,\by) \in \dom \phi$.
For \emph{2.}, it suffices to use the equality in \eqref{lineq}  and the
second conclusion of Lemma \ref{quad}. Conditions \emph{3.}~and 
\emph{4.}~ follow from the computation $\phi_{y^i}(\bt, \bx, \by) = E^i(\bx,\by) + \delta^{ii'} - 1 = \delta^{ii'}$. To handle the
 constraint $\phi_{y^{i_0}}(\bt,\by, \bx) = -1$ we use 
 $E^{i_0}(x,y)^{-1}$ instead of $E^{i_0}(x,y)$ in the definition of
 $\phi$ and repeat the argument.
\end{proof}

\begin{proof}[Proof of Theorem \ref{main}]
By design, each solution $Y$ to \eqref{bsde} is a local
$f$-martingale, so we focus on the opposite implication and proceed
by contradiction: we assume that $Y$ is a local
$f$-martingale which is not a solution to
\eqref{bsde}.

By assumption, $Y$ is an \ito{}-process, i.e.,
admits a decomposition of the form  
\[ Y_t = Y_0 - \int_0^t g_s\, ds + \int_0^t Z_s \cdot dB_s,\] 
for some locally integrable $g$ and $Z\in \sP^2$. Since $Y$ does not solve
\eqref{bsde}, the processes $\int_0^{\cdot} g^{i_0}_s ds$ and $\int_0^{\cdot}
f^{i_0}(s,Y_s,Z_s) ds$ are not indistinguishable for some $i_0 \in 
\{1,2,...,n\}$, which means that 
\[ \lpB{g^{i_0} \neq f^{i_0}(\cdot, Y,Z)} \coloneqq \lpB{ (t,\omega) \, : \, g^{i_0}_t(\omega)
\neq f^{i_0}(t, Y_t(\omega), Z_t(\omega))} >0,\] 
where $\ld$ denotes the Lebesgue measure on $[0,T]$. Let us assume
first that 
\[ \lp{g^{i_0}>f^{i_0}(\cdot, Y,Z)}>0 , \]
and thus
\[ \lp{g^{i_0}>f^{i_0}(\cdot, Y,Z)+\delta}>0 \text{  for some
$\delta>0$.} \]
  By considering a countable partition, we can find a
cube $C_1 \subseteq [0,T] \times \R^d \times \R^n \times \R^{n \times d}$
with side length $T$ and a constant $K > 0$ such that 
\begin{align*} \lpB{ (\cdot, B,Y,Z)
\in C_1, \, g^{i_0} > f^{i_0}(\cdot, Y,Z)+\delta, \, |g| < K, |f(\cdot, B, Y, Z)| < K}>0.  
\end{align*} 
Partitioning $C_1$ yields a cube $C_2$ of side length $\frac{T}{2}$
with the
same property, and, iteratively, we can construct a nested sequence of
cubes $C_n$ whose diameters go to zero and such that 
\begin{align*}
\lpB{ (\cdot, B,Y,Z)
\in C_n, \, g^{i_0} > f^{i_0}(\cdot, Y,Z)+\delta, \, |g| < K, |f(\cdot, B, Y, Z)| < K}>0
\end{align*} 
We choose a point $(\bt,\by, \bx, \bz) \in \cap_n C_n$, and note that any
neighborhood $\sV$ of $(\bt,\by, \bx, \bz)$ contains some $C_n$, and hence
satisfies
\begin{align} 
\label{V-neigh}
\lpB{ (\cdot, B,Y,Z)
\in \sV, \, g^{i_0} > f^{i_0}(\cdot, Y,Z)+\delta, \, |g| < K, |f(\cdot, B, Y, Z)| < K}>0.
\end{align} 
Given the constant $\delta$ and the point $(\bt,\bx,\by, \bz)$ 
constructed above,
Lemma \ref{mainlem} guarantees the existence of an $f$-subharmonic function
$\phi$ with $(\bt,\bx,\by) \in \dom \phi$ such that 
\[ \phi_{y^{i_0}}(\bt, \bx,\by) = 1, \phi_{y^i}(\bt,\bx,\by) = 0 \text{ for } i \neq {i_0}, \eand \sL^f \phi (\bt,\bx,\by;\bz) < \delta/16.\]
By shrinking $\dom \phi$ to a smaller neighborhood of $(\bt,\bx,\by)$, if
necessary,  we can assume further that $\phi_{y^{i_0}} > \frac{1}{2}$ on $\dom \phi$, $\phi_{y^i} < \frac{\delta}{16 n K}$ on $\dom \phi$ for $i \neq k$, and we can find a  constant $r>0$   such that $\sL^f \phi \leq
\delta/8$ on $\dom \phi \times B_r(\bz)$, where $B_r(\bz)$ is the open ball in
$\R^{n \times d}$ of radius $r$ around $\bz$.

Next, we select a neighborhood $V$ of $(\bt,\bx,\by)$ with $\Cl V \subseteq \dom\phi$
so that the set $W = \rt\times\rx\times\ry\setminus \Cl V$ satisfies
\[ W \cup \dom\phi = [0,T] \times \R^{d} \times \R^n \eand W
\cap V = \emptyset.\] Here we denote by $\Cl V$ the closure of $V$. The process $(t,B_t,Y_t)$ is continuous
and $\set{\dom \phi, W}$ is  an open cover of $[0,T] \times \R^{d} \times \R^n$
so there exists a nondecreasing 
sequence $\{T_k\}_{k\in\N}$ of $[0,T]$-valued 
stopping times with the following properties 
(see, e.g., \cite[Lemma 3.5, p.~22]{Eme89}) :
\begin{enumerate}
  \item $\PP[ T_k < T] \to 0$, as $k\to\infty$
  \item the path $(t,B_t,Y_t)$ lies entirely in $\dom \phi$ or entirely in $W$ 
    on each stochastic interval $[T_k, T_{k+1})$.
\end{enumerate}
Since $[0,T] =
\cup_k [T_k, T_{k+1})$ a.s., equation \eqref{V-neigh} above guarantees the 
existence of an index $k_0\in\N$ such that
\begin{multline}
      \lpB{ \big\{(t,\omega) : t \in [T_{k_0}(\omega), T_{k_0+1}(\omega)) \big\} \cap \Bset{ (\cdot, B,Y,Z)
\in V \times B_r(\bz), \\ g^{i_0} > f^{i_0}(\cdot, Y,Z)+\delta, \, 
|g| < K, |f(\cdot, B, Y, Z)| < K}} >0
\end{multline}
 
Moreover, because $W  \cap V = \emptyset$, the stopping times 
\[ \tau_1 = T_{k_0} \eand
\tau_2 = T_{k_0+1} \wedge \inf \{t \geq \tau_1 :
      (t,B_t,Y_t) \notin \dom \phi\} \]
have the property that $(t,B_t, Y_t) \in \dom\phi$ on 
$[\tau_1, \tau_2)$ as well as
\begin{align}
\label{tau-sK}
\lpB{ \big\{ (t, \omega) : t \in [\tau_1(\omega), \tau_2(\omega))\big\} \cap \sK} >0 
\end{align}
where
\begin{multline}
\label{sK}
    \sK= \Big\{ (\cdot, B,Y,Z)
\in V \times B_r(\bz), \, : \,  g^{i_0} > f^{i_0}(\cdot, Y,Z)+\delta,
\, |g|
< K, \\ 
|f(\cdot, B, Y, Z)| < K\Big\}.
\end{multline}
The drift  of the process $\phi(t,B_t,Y_t)$ can be written as
\begin{multline*}
 \mu_t = \sL^f \phi(t,B_t, Y_t; Z_t) + \phi_{y^{i_0}}
(t,B_t,Y_t)\Big
(f^{i_0}(t,B_t, Y_t) - g_t^{i_0}\Big) + \\ + \sum_{i \neq i_0} \phi_
{y^i}
(t,B_t,Y_t)\Big(f^i(t,B_t, Y_t) - g_t^i \Big)
\end{multline*}
and using the bounds on its terms enforced by the construction of the domain 
$\dom \phi$ above, we find that
\begin{align}
\label{mu-leq}
\mu_t \leq  \delta/8  -  \delta/2 + \delta/8 = - \delta/4,\  \lop\text{-a.e.~on }\sK.
\end{align}
On the other hand, 
$Y$ is a local $f$-martingale so 
\begin{align}
\label{mu-geq}
\mu_t \geq 0,\  \lop\text{-a.e.~on } [\tau_1, \tau_2),
\end{align}
which, thanks to \eqref{tau-sK}, contradicts \eqref{mu-leq}.

To deal with the case  $\lp{g<f(\cdot,Y,Z)}>0$, we use  essentially
the same argument, but with $\phi_{y^{i_0}}(\bt,\bx,\by)=-1$.
\end{proof}
In anticipation of the closure result Theorem \ref{ucplimits} below, we wish to identify conditions under which the assumption $Y$ is an \ito \,process can be removed in the statement of Theorem \ref{main}. In general, it seems to difficult to rule out the possibility that a local $f$-martingale has singular drift when $n > 1$. But, when the driver $f$ has some additional structure, we can indeed guarantee that local $f$-martingales are It\^o processes. One structure which works is (a local version of) the condition (AB), introduced in \cite{xing2018}. We recall the definition here, adapted slightly to meet our needs. 

\begin{defn}
A driver $f$ is said to \define{satisfy the condition (AB)} there exists a finite collection $\sam = (a_1,\dots, a_M)$ of
  vectors in $\R^n$ such that
  \begin{enumerate}
    \item $a_1,\dots, a_M$ positively span $\R^n$ 
    \item $ a_m^T f(t,y, z) \leq k + \tfrac{1}{2} \abs{a_m^T z}^2$
    for each $m$, for all $t,y,z \in [0,T] \times  \R^n \times \R^{n \times d}$. 
  \end{enumerate}
  We say that $f$ \textbf{satisfies the condition (AB) locally} if for each $k \in \N$, the driver $f^k(t,y,z) \coloneqq f(t,\pi_k(y),z)$ satisfies the condition (AB), where $\pi_k$ is the orthogonal projection onto the ball of radius $k$ in $\R^n$. 
\end{defn}

\begin{rmk} \label{rmk:dimone}
When $n = 1$, Assumption \ref{assumptions} already implies that $f$ satisfies the condition (AB) locally. Thus when $n = 1$, the hypothesis that $f$ satisfies the condition (AB) locally can be omitted in the statements of the following Lemma and Corollary;. 
\end{rmk}

\begin{lem} \label{lem:absemimart}
Suppose that $f$  satisfies the condition (AB) locally . Then each local $f$-martingale is automatically an It\^o process. 
\end{lem}

\begin{proof}
By a localization argument, we may assume that $f$ satisfies the condition (AB). Let $k$, $\sam$ be as given in the definition of the condition (AB). For each $m$, define $\phi : [0,T] \times \R^n \to \R$ by 
\begin{align*}
    \phi_m(t,x,y) = \phi_m(t,y) = \exp(2 a_m^T y + 2kt). 
\end{align*}
A computation shows that $\sL^f \phi_m(t,x,y,z) = \exp(2 a_m^T y + 2kt) \big( -2 a_m^T f(t,y,z) + 2k + 2 |a_m^T z|^2\big)$. Since $a_m^T f(t,y,z) \leq k + |a_m^T z|^2$, we see that 
\begin{align*}
    -2 a_m^T f(t,y,z) + 2k + 2 |a_m^T z|^2 \geq 0, 
\end{align*}
and so $\phi_m$ is an $f$-subharmonic function for each $m$. Since $Y$ is a local $f$-martingale by hypothesis, it follows that $\phi_m(t,Y) = \exp(2a_m^T Y + 2kt)$ is a local submartingale for each $m$, hence a semimartingale. The map $(t,y) \mapsto (t,\exp(2y + 2kt))$ is invertible with $C^2$ inverse, so we conclude that $a_m^T Y$ is a semimartingale for each $m$. Since $\sam$ spans $\R^n$, this shows that $Y$ is a semimartingale. 
\medskip
Now that we know $Y$ is a semimartingale, we represent as $Y = Y_0 - A_t + \int_0^t Z_s \cdot dB_s$, for a constant $Y_0$, continuous finite-variation process $A$, and $Z \in \sP^2$. The finite-variation part of the process $\phi_m(t,Y)$ is given by  $D - E$, where
\[
D =  2 \int_0^{\cdot} \exp(2a_m^T Y + kt)\big(|a_m^T Z|^2 + k\big) dt, \, \, E = -2 \int_0^{\cdot} \exp(2a_m^T Y + kt) d(a_m^T A). 
\]
Since $\phi_m(t,Y)$ is a submartingale, we see that $E \preceq D$ a.s. Thus we can find for each $m$ an absolutely continuous process $H^m$ such that $a_m^T A \preceq H^m$. Since $\sam$ positively spans $\R^n$, for each basis vector $e_i$ of $\R^n$, we may find non-negative scalars $\lambda_1,...,\lambda_M$ such that $e_i = \sum_m \lambda_m a_m$, and so 
\begin{align*}
    A^i = \sum_{m} \lambda_m a_m^T A^i \preceq \sum_m \lambda_m H^m. 
\end{align*}
But the same argument works for $-e_i$, i.e. we may find positive constants $\lambda_m'$ such that $-e_i = \sum_m \lambda_m'$, which leads to  
\begin{align*}
    - A^i = \sum_m \lambda_m (a_m')^T A^i \preceq \sum_m \lambda_m' H^m. 
\end{align*}
We conclude that 
\begin{align*}
    - \sum_m \lambda_m' H^m \preceq A^i \preceq \sum_m \lambda_m H^m, 
\end{align*}
a.s., and so for each $i$, $A^i$ has absolutely continuous paths, a.s. The argument is completed by applying standard theory to produce a progressive density for $A$ (see, e.g., \cite[Proposition
3.13, p.~30]{JacShi03}). 
\end{proof}

Thanks to Theorem \ref{main}, we get the following Corollary. 

\begin{cor} \label{cor:ab}
If $f$  satisfies the condition (AB) locally and Assumption \ref{assumptions} holds, then a process $Y$ is a local $f$-martingale if and only if it solves \eqref{bsde}. 
\end{cor}

\begin{rmk}
We note that other structural conditions can be used to guarantee that local $f$-martingales are automatically \ito \,  processes. For example, if $f$ is triangular in the sense of Definition 2.10 of \cite{xing2018}, then an argument very similar to the one appearing in the proof of Lemma \ref{lem:absemimart} shows that any local $f$-martingale is an It\^o process.
\end{rmk}

\section{Applications}

\subsection{Stability of Solutions to BSDE} The set of solutions to
\eqref{bsde} shares several properties with the set of continuous
local martingales. We focus
here on one which concerns stability under ucp
convergence and show how it follows from Lemma \ref{cor:ab}.  We remind the reader that a sequence
$\sq{Y^k}$ of continuous processes converges \emph{uniformly on
compacts in probability (in ucp)}, denoted by $Y^k \toucp Y$, if
\[ \sup_{t\in [0,T]} \abs{Y^k_t - Y_t} \to 0 \text{ in probability.}\]

\begin{thm}\label{ucplimits}
Suppose that Assumption \ref{assump2} holds, and that $f$ satisfies the condition (AB) locally. Then the set of solutions to \eqref{bsde} is closed under ucp convergence.
\end{thm}

\begin{proof}
Let $\{Y_k\}_{k \in \N}$ be a sequence of solutions of \eqref{bsde} such that
$Y^k \toucp Y$. Clearly the limit process $Y$ is continuous and
adapted. To show that it solves \eqref{bsde} via Corollary \ref{cor:ab}, we 
pick $\phi \in S^f$ and a pair of stopping times $\tau_1,
\tau_2$ such that $(t,Y_t,B_t) \in \dom \phi$ for $t\in
[\tau_1, \tau_2)$, a.s. In fact, it is clear from the proofs of Theorem \ref{main} and Lemma \ref{lem:absemimart} we need only consider $\phi$ such that there exists $\tilde{\phi} \in S^f$ with $\Cl \dom \phi \subset \dom \tilde{\phi}$, $\phi = \tilde{\phi}$ on $\dom \phi$, and $\phi$ is Lipschitz on $\dom \phi$. 

For $k\in\N$, let the stopping time $S_k$ be defined by  \[ S_k \coloneqq
\tau_2 \wedge \inf \{t \geq \tau_1 : (t,B_t, Y_t^k) \notin
\dom\tphi\},\]
so that $(t,B, Y^k_t) \in \dom \tphi$ for $t\in [\tau_1, S^k)$, and,
consequently $\phi(t,B_t, Y^k_t)$ is a local
submartingale on $[\tau_1, S^k)$. 

Since $\Cl \dom\phi \subseteq \dom\tphi$ and
$(t,B_t,Y^k_t) \toucp (t,B_t,Y_t)$, we necessarily have 
$\pr[S_k \ne \tau_2] \to
0$. Together with the fact that $(t,B_t,Y_t^k) \toucp (t,B_t,Y_t)$, this implies that
\[ \lmr{\tau_1}{ (t,B_t,Y^k_t)}{S^k} \toucp \lmr{\tau_1}{(t,B_t,Y_t)}{\tau_2}.\]
Now since $\phi$ is Lipschitz, it follows that 
\[ \lmr{\tau_1}{ \phi(t,B_t, Y^k_t)}{S^k} \toucp \lmr{\tau_1}{
\phi(t, B_t, Y_t)}{\tau_2}.\]
It remains to use the fact that the class of local submartingales is closed under the ucp convergence. 
\end{proof}

\subsection{A template for existence}
We turn to another application of our main result Theorem \ref{main},
namely to the existence of solutions to \eqref{bsde} in dimension $1$. In order to highlight where the one-dimensional structure is used (and hence what difficulties must be overcome to extend to systems), we start by proving a sufficient condition for existence which holds in any dimension. In the remainder of the paper, we will work with the following standard assumption.
\begin{assumption}[Growth and regularity of the driver] 
\label{assump2}

The driver $f:\rt\times\ry\times \rz \to \ry$ is jointly continuous and
there exists a constant $M$ such that 
\begin{enumerate}
  \item (Regularity) for all 
    $t \in [0,T]$ as well as $(y,z), (y',z') \in \ry\times \rz$, we have
     \[
       \abs{ f(t,y',z') - f(t,y,z) } 
       \leq 
       M\abs{y - y} + M(1 + |y| + |y'| + |z| + |z'|) \abs{z - z'}
\]
    \item (Growth) for all $t,y,z\in\rt\times \ry \times \rz$, 
      \[ \abs{f(t,y,z)} \leq  + M(1 + \abs{y} + \abs{z}^2\big).\]
\end{enumerate} 
\end{assumption}
The main idea behind our
proof is inspired by the work of Darling
\cite{Darling} in stochastic differential geometry, which is, in turn,
based on ideas of Picard \cite{Pic94}.
We start in a standard way and build a sequence of approximate
equations with Lipschitz coefficients:
let $\pi_R$ denote the
orthogonal projection $\pi_R:\R^d \to \R^d$ onto the
closed ball of radius $R$ centered at the origin.  
We set
\[ \fk=f(\cdot,\cdot, \pi_k(\cdot)), \]
and note that $\fk$ is 
uniformly Lipschitz in $(y,z)$ and that it satisfies Assumption 
\ref{assump2} with the same constant $M$ as $f$.
Standard theory (see, e.g. \cite[Theorem
4.3.1, p.~84]{Zhang}) guarantees existence of a unique solution $(\ryk,
\Zk)$ to 
\begin{align}
\label{Yn-fn}
d\ryk = - \fk(\cdot,\ryk, \Zk)\, dt + \Zk\cdot dB,\  \ryk_T = \xi,
\end{align}
with $\Zk \in \sL^2$. 

\begin{defn}
 We will say that the sequence $(Y^k,Z^k)$ is a \textbf{bounded approximation scheme} if $\sup_k \big(\norm{Y^k}_{\sS^{\infty}} + \norm{Z^k}_{\bmo}\big) < \infty$. We call the sequence $(Y^k,Z^k)$ \textbf{uniformly integrable in probability} if $\pr[\int_0^T |Z^k|^21_{|Z^k|^2 > 1/k} dt > k] \to 0$ as $k \to \infty$. 
\end{defn} 

\begin{defn}
We say that the driver $f$ is \textbf{ucp stable} if the following holds: if $(Y^k,Z^k)$ is a sequence of solutions to \eqref{bsde} with $\sup_k \norm{Z^k}_{\bmo} < \infty$, and $\{Y^k_T\}_{k \in \N}$ is Cauchy in $\LL^p$ for all $1 \leq p < \infty$, then $\{Y^k\}_{k \in \N}$ is Cauchy with respect to ucp convergence. 
\end{defn}
 
\begin{prop}[Template for existence] \label{template}
Suppose that $f$ is ucp stable, satisfies Assumption \ref{assump2}, and satisfies the condition (AB) locally. Suppose further that the sequence $(Y^k,Z^k)$ constructed above is bounded and uniformly integrable in probability. Then, there exists a solution $(Y,Z) \in \sS^{\infty} \times \bmo$ to \eqref{bsde}. 
\end{prop}

Before proving Proposition \ref{template}, we use the solutions $(\ryk, \Zk)$ defined above to introduce the following
sequence
of \emph{forward} SDEs with random coefficients:
\begin{equation}
  \label{SDE}
  \rxk_0 = \ryk_0,\ d\rxk = -\fk(\cdot,\rxk,\Zk )\, dt + \pi_k
  (\Zk)\cdot \, dB.
\end{equation}
Standard theory (see, e.g., \cite[Theorem 3.3.1, p.~68]{Zhang} and \cite[Theorem 3.4.3, p.~72]{Zhang})
guarantees that \eqref{SDE} 
has a unique strong solution $\rxk$ which lies in $\cap_{p \geq 1} \mathcal{S}^p$. 
A key observation here is that, since $\pi_k$ is a projection, the
pair $(\rxk, \pi_k(\Zk))$ satisfies the same equation \eqref{Yn-fn} as $(\ryk, \Zk)$. Of
course, the terminal value $\rxk_T$
of $\rxk$ 
will, in general, differ from $\ryk_T=\xi$. The following Lemma follows from standard stability results for SDEs, together with the assumption that $(Y^k,Z^k)$ is bounded and uniformly integrable in probability.  
\begin{lem} \label{stability}
  Under the hypotheses of Proposition \ref{template}, and with $\ryk$ and $\rxk$ defined in \eqref{Yn-fn} and \eqref{SDE}
  above,  we have  \[ \ryk - \rxk \to 0 \text{ in $\sS^p$ for each
  $p \in [1,\infty)$.} \]
    In particular $\rxk_T \to \xi$, in $\LL^p$ for each $p\in
    [1,\infty)$. 
\end{lem}
Now we report the proof of Proposition \ref{template}. 
\begin{proof}[Proof of Proposition \ref{template}.]\ 
Lemma \ref{stability}, together with the assumption that $f$ is ucp stable, implies that $\{X^k\}_{k \in \N}$ is Cauchy with respect to ucp convergence. Thus we identify a process $Y$ such that $X^k \toucp Y$. Since $X^{k}_T \to \xi$ in $\LL^{p}$, it necessarily follows that $Y_T = \xi$. Applying Theorem \ref{ucplimits} shows that $Y$ is a solution to \eqref{bsde} with terminal condition $\xi$. Standard estimates then give the desired regularity. 
\end{proof}

\begin{rmk} \label{rmk:template}
There are various conditions which guarantee that the sequence $(Y^k,Z^k)$ is bounded, even for systems. For example, if $f$ satisfies the condition (AB), then $(Y^k,Z^k)$ is bounded - see the proof of Theorem 2.14 in \cite{xing2018}. However, when $n > 1$ we do not currently have simple conditions which guarantee $f$ is ucp stable or that the sequence $(Y^k,Z^k)$ is uniformly integrable in probability - this is an issue we hope to address in future research. However, we note that the hypotheses of Proposition \ref{template} can be verified when $f$ is the geometric driver appearing in \eqref{martdriver}, at least under certain geometric conditions on the corresponding connection - this is exactly the program carried out in \cite{Darling}. It is also possible to verify the hypotheses of Proposition \ref{template} when $n = 1$, again under appropriate technical assumptions. Indeed, when $n = 1$ ucp stability can be shown through a change of measure argument (see for example the proof of Proposition 2.3 in \cite{Briand-Elie}), while uniform integrability in probability can be obtained by adapting the coupling argument in \cite{Darling}.
\end{rmk}

\bibliographystyle{amsalpha}
\bibliography{bsde}

\end{document}